\documentclass[12pt]{amsart}
\usepackage{amsmath,amssymb,amsfonts,amsthm,amstext,verbatim}
\usepackage{url,xspace,enumerate,color,graphicx,stmaryrd}

\newcommand{\Z}{{\mathbb Z}}
\newcommand{\R}{{\mathbb R}}

\renewcommand{\P}{{\mathbb P}}
\newcommand{\E}{{\mathbb E}}
\renewcommand{\L}{{\mathbb L}}

\newcommand{\dof}{\bf\boldmath}
\newcommand\down{\kern5pt\downarrow\kern-4pt}

\newcommand\oo{\infty}
\newcommand\sm{\setminus}

\newcommand\rad{\text{\rm rad}\,}
\renewcommand\b{\beta}

\newcommand\resp{respectively}

\newcommand\eps{\epsilon}
\renewcommand\a{\alpha}

\newcommand\pc{p_{\text{\rm c}}}
\newcommand\pco{\vec p_{\text{\rm c}}}
\newcommand\pe{p_{\text{\rm e}}}
\newcommand\tg{\theta_{\text{\rm g}}}
\newcommand\pg{p_{\text{\rm g}}}
\newcommand\ta{\theta_{\text{\rm a}}}
\newcommand\pa{p_{\text{\rm a}}}
\newcommand\tah{\theta_{\text{\rm a}}^H}
\newcommand\pah{p_{\text{\rm a}}^H}

\newcommand\q{\quad}
\newcommand\qq{\qquad}
\newcommand\adm{admissible}

\newcommand\comp{\overline}
\newcommand\toa{\to_{\text{\rm a}}}
\newcommand\too{\to_{\text{\rm oo}}}
\newcommand\toah{\toa^H}

\newtheorem{thm}{Theorem}
\newtheorem{lemma}[thm]{Lemma}
\newtheorem{prop}[thm]{Proposition}
\newtheorem{cor}[thm]{Corollary}
\newtheorem{question}{Open Question}

\newenvironment{letlist}{\begin{list}{{\rm(\alph{mycount})}}%
   {\usecounter{mycount}\labelwidth=1cm\itemsep 0pt}}{\end{list}}

\newcounter{mycount}

\begin{document}
\title{Plaquettes, Spheres, and Entanglement}
\author[Grimmett]{Geoffrey R.\ Grimmett}
\address[G.\ R.\ Grimmett]{Statistical Laboratory, Centre for
Mathematical Sciences, Cambridge University, Wilberforce Road,
Cambridge CB3 0WB, UK} \email{g.r.grimmett@statslab.cam.ac.uk}
\urladdr{http://www.statslab.cam.ac.uk/$\sim$grg/}

\author[Holroyd]{Alexander E.\ Holroyd}
\address[A.\ E.\ Holroyd]{Microsoft Research, 1 Microsoft Way,
Redmond WA 98052, USA; and Department of Mathematics, University of British
Columbia, 121--1984 Mathematics Road, Vancouver, BC V6T 1Z2, \linebreak
Canada} \email{holroyd at math.ubc.ca}
\urladdr{http://research.microsoft.com/$\sim$holroyd/}

\begin{abstract}
The high-density plaquette percolation model in $d$ dimensions
contains a surface that is homeomorphic to the $(d-1)$-sphere
and encloses the origin. This is proved by a path-counting
argument in a dual model. When $d=3$, this permits an improved
lower bound on the critical point $\pe$ of entanglement
percolation, namely $\pe\ge \mu^{-2}$ where $\mu$ is the
connective constant for self-avoiding walks on $\Z^3$.
Furthermore, when the edge density $p$ is below this bound, the
radius of the entanglement cluster containing the origin has an
exponentially decaying tail.
\end{abstract}

\date{10 February 2010 (revised 15 August 2010)}


\maketitle

\section{Introduction and results}\label{sec:intro}

The {\dof plaquette percolation} model is a natural dual to
bond percolation in two and more dimensions.  Let $\Z^d$ be the
integer lattice; elements of $\Z^d$ are called {\dof sites}.
For any site $z$, let $Q(z):=[-\tfrac12,\tfrac12]^d+z$ be the
topologically closed unit $d$-cube centred at $z$.  A {\dof
plaquette} is any topologically closed unit $(d-1)$-cube in
$\R^d$ that is a face of some $Q(z)$ for $z\in\Z^d$.  Let
$\Pi_d$ be the set of all plaquettes. For a set of plaquettes
$S\subseteq\Pi_d$, we write $[S]:=\bigcup_{\pi\in S}\pi$ for
the associated subset of $\R^d$.  In the plaquette percolation
model with parameter $p\in[0,1]$, each plaquette of $\Pi_d$ is
declared {\dof occupied} with probability $p$, otherwise {\dof
unoccupied}, with different plaquettes receiving independent
states; the associated probability measure is denoted $\P_p$.

The $\ell^s$-norm on $\R^d$ is denoted $\|\cdot\|_s$. A {\dof
sphere} of $\R^d$ is a simplicial complex, embedded in $\R^d$,
that is homeomorphic to the unit sphere
$\{x\in\R^d:\|x\|_2=1\}$. By the generalized Sch\"onflies
theorem, the complement in $\R^d$ of a sphere has a bounded and
an unbounded path-component, which we call respectively its
{\dof inside} and {\dof outside}. The ($\ell^1$-){\dof radius}
of a set $A\subseteq\R^d$ with respect to the origin $0\in\R^d$
is
$$
\rad A = \text{rad}_0
A:=\sup\{\|x\|_1:x\in A\}.
$$

Let $\mu_d$ be the {\dof connective constant} of $\Z^d$,
given as
$$
\mu_d:=\lim_{k\to\infty} \sigma(k)^{1/k},
$$
where $\sigma(k)$ is the number of ($\ell^1$-)nearest-neighbour
self-avoiding paths from the origin with length $k$ in $\Z^d$;
it is a straightforward observation that $\mu_d\in[d,2d-1]$,
and stronger bounds may be found, for example, in \cite{ponitz-tittman}.

\begin{thm}\label{sphere}
Let $d \ge 2$, and consider the plaquette percolation model.  If
$p<\mu_d^{-2}$ then almost surely there exists a finite set $S$
of unoccupied plaquettes whose union $[S]$ is a sphere with $0$
in its inside.  Moreover $S$ may be chosen so that
$$\P_p\big(\rad[S]\geq r\big)\leq C \alpha^r,\qquad r>0,$$
for any $\alpha\in(\mu_d p,1)$, and some
$C=C(p,d,\alpha)<\infty$.
\end{thm}

When $d=2$, the first assertion of Theorem \ref{sphere} amounts
to the well known fact that there exists a suitable circuit of
unoccupied bonds of the dual lattice (see, e.g., \cite{g2}).
The result is more subtle in higher dimensions.

When $d=3$, Theorem \ref{sphere} has an application to
entanglement percolation, which we explain next.  Define a
{\dof bond} to be the topologically closed line segment in
$\R^d$ joining any two sites $x,y\in\Z^d$ with $\|x-y\|_1=1$.
Let $\L_d$ be the set of all bonds. In the {\dof bond
percolation} model, each bond is declared {\dof occupied} with
probability $p$, otherwise {\dof unoccupied}, with the states
of different bonds being independent. For a set of bonds $K$,
write $[K]:=\bigcup_{e\in K} e$.  We say that $K$ contains a
site $x$ if $x\in[K]$.

We say that a sphere $Z\subset \R^d$ {\dof separates} a set
$A\subset\R^d$ if $A$ intersects both the inside and the
outside of $Z$, but not $Z$ itself.  (We write $A \subset B$ if
$A \subseteq B$ and $A \ne B$.) Let $d=3$.  We say that a set
of bonds $K\subseteq\L_3$ is $1$-{\dof entangled} if no sphere
of $\R^3$ separates $[K]$.  The idea of this definition is that
a $1$-entangled set of bonds, if made of string or elastic,
cannot be continuously ``pulled apart''. Any connected set of
bonds is evidently $1$-entangled. The simplest disconnected set
that is $1$-entangled consists of two linked loops. The prefix
``$1$'' reflects the fact that other natural definitions of
entanglement are possible; see \cite{gh} and the discussion in
Section \ref{sec:rem} for more details. Entanglement of sets of
bonds is intrinsically a three-dimensional issue, and therefore
we shall always take $d=3$ when discussing it.

In the bond percolation model in $d=3$, let $\eta^1(p)$ be the
probability that there exists an infinite $1$-entangled set of
occupied bonds  containing the origin $0$, and define the
$1$-entanglement {\dof critical probability} $\pe^1:=\sup\{p:
\eta^1(p)=0\}$.  The maximal $1$-entangled set of occupied
bonds containing a site $x$ is called the $1$-entanglement
{\dof cluster} at $x$.

\begin{cor}\label{ent} The $1$-entanglement critical probability
in three dimensions satisfies
$$
\pe^1\geq \mu_3^{-2}.
$$
Moreover, if $p<\mu_3^{-2}$, the $1$-entanglement cluster
$E$ at the origin satisfies
$$
\P_p\big(\rad [E]\geq r\big)\leq C \alpha^r,\qquad r>0,
$$
for any $\alpha\in(\mu_3 p,1)$, and some $C=C(p,\alpha
)<\infty$.
\end{cor}

The connective constant of $\Z^3$ satisfies the rigorous bound
$\mu_3 \leq 4.7387$ (see \cite{ponitz-tittman}). Therefore, Corollary
\ref{ent} gives
$$
\pe^1 \geq 0.04453\cdots > \frac{1}{23}.
$$
This is a significant improvement on the previous best lower
bound of \cite{atapour-madras}, namely $\pe^1 > 1/597$, which
in turn substantially improved the first non-zero lower bound,
$\pe^1 \geq 1/15616$, proved in \cite{h-ent}.  In each case,
the improvement is by a factor of approximately $26$.

In Section \ref{sec:rem} we discuss some history and background
to our results.  In Section \ref{sec:pfs} we prove Theorem
\ref{sphere} and Corollary \ref{ent}.  If $p$ satisfies the
stronger bound $p<(2d-1)^{-2}$, we shall see that our methods
yield versions of these results with explicit formulae for the
constants $C$ and $\alpha$.  In Section \ref{sec:crit} we
consider the critical value of $p$ associated with the event in
Theorem~\ref{sphere}, and its relationship to certain other critical
values.

\section{Remarks}\label{sec:rem}

\subsection{Duality}
To each bond $e\in\L_d$ there corresponds a unique plaquette
$\pi(e)\in\Pi_d$ that intersects $e$.  It is therefore natural
to couple the bond and plaquette percolation models with common
parameter $p$ in such a way that $\pi(e)$ is occupied if and
only if $e$ is occupied.  If $p$ is less than the critical
probability $\pc$ for standard bond percolation (see, e.g.,
\cite{g2}), the connected component of occupied bonds at the
origin is almost surely finite, and it is a straightforward
consequence that there exists a finite set of unoccupied
plaquettes whose union encloses the origin (i.e., the origin
lies in some bounded component of its complement).  Indeed,
such plaquettes may be chosen so as to form a `surface'
enclosing the origin (although precise definition of such an
object requires care). However, such a surface might be
homeomorphic to a torus, or some other topological space.  It
is a key point of Theorem \ref{sphere} that the surface $[S]$
is a sphere.

\subsection{Entanglement}
Entanglement in three-dimensional percolation was
first studied, in a partly non-rigorous way, in
\cite{kantor-hassold} (some interesting remarks on the subject
appeared earlier in \cite{accfr}). The rigorous theory was
systematically developed in \cite{gh}, and further rigorous
results appear in
\cite{aizenman-g,haggstrom-ent,h-ent,h-er,h-ent-ineq}. A discussion
of physical applications of entanglement percolation may be found
in \cite{atapour-madras}.

As mentioned in Section \ref{sec:intro}, there are several
(non-equivalent) ways of defining the property of entanglement
for infinite graphs. One of these, namely $1$-entanglement, was
presented in that section, and a second follows next.  We say
that a set of bonds $K\subseteq \L_3$ is $0$-{\dof entangled}
if every finite subset of $K$ is contained in some finite
$1$-entangled subset of $K$. It was shown in \cite{gh} that the
notions of $0$-entanglement and $1$-entanglement are extremal
members of a certain class of natural candidate
definitions, called {\em entanglement systems}, and furthermore
that these two entanglement systems correspond (respectively)
in a natural way to free and wired boundary conditions.

By combining inequalities of \cite{aizenman-g,gh,h-ent}, we
find that
$$
0<\pe^1\leq \pe^\mathcal{E}\leq \pe^0<\pc<1,
$$
where $\pe^0$, $\pe^1$, and $\pe^\mathcal{E}$ are the critical
probabilities for $0$-entanglement, $1$-entanglement, and for
an arbitrary entanglement system $\mathcal{E}$, respectively.
The inequality $\pe^0\leq \pc$ reflects the straightforward
fact that every connected set of bonds is $0$-entangled.  It
was strengthened to the strict inequality $\pe^0 < \pc$ in
\cite{aizenman-g,h-ent-ineq}. In \cite{kantor-hassold} it was
argued on the basis of numerical evidence that $\pc-\pe\approx
1.8\times 10^{-7}$, for a certain notion of `entanglement
critical probability' $\pe$.  It is an open question to decide
whether or not $\pe^0 = \pe^1$.

\subsection{Spheres, lower bounds, and exponential decay}
The inequality $\pe^1>0$ expresses the fact that, for a
sufficiently small density $p$ of occupied bonds, there is no
infinite entangled set of bonds.  Prior to the current paper,
proofs of this seemingly obvious statement have been very
involved.

The proof in \cite{h-ent} employs topological arguments to show
that, for $p<1/15616$, almost surely the origin is enclosed by
a sphere that intersects no occupied bond.  The argument is
specific to three dimensions, and does not resolve the question
of the possible existence of a sphere of unoccupied plaquettes
enclosing the origin. (See \cite{gh} for more on the
distinction between spheres intersecting no occupied bond, and
spheres of unoccupied plaquettes.) In \cite{gh}, related
arguments are used to show that, for sufficiently small $p$,
the radius $R$ of the $1$-entanglement cluster at the origin
has `near-exponential' tail decay in that
$$
\P(R>r)<\exp\bigl(-c r/\log\cdots\log r\bigr)
$$
for an arbitrary iterate of the logarithm, and for some $c > 0$
depending on $p$ and the number of logarithms.

In the recent paper \cite{atapour-madras}, the above results
are substantially improved in several respects.  The lower
bound on the critical point is improved to approximately
$\pe^1>1/597$, and it is proved also that the radius of the
$1$-entanglement cluster at the origin has exponential tail
decay for $p$ below the same value.  The key innovation is a
proof of an exponential upper bound on the number of possible
$1$-entangled sets of $N$ bonds containing the origin, thereby
answering a question posed in \cite{gh}.  The method of proof
is very different from that of \cite{h-ent}.

In the current article, we improve the proofs mentioned above
in several regards. The lower bound on the critical point is
further improved to approximately $\pe^1>1/23$, and we
establish exponential decay of the $1$-entanglement
cluster-radius at the origin for $p$ below this value. We prove
the existence of a sphere of unoccupied plaquettes enclosing
the origin (rather than just a sphere intersecting no occupied
bond), and we do so for all dimensions.  Finally, our proofs
are very simple. Our methods do not appear to imply the key
result of \cite{atapour-madras} mentioned above, namely the
exponential bound for the number of entangled sets containing the origin.

\section{Proofs}\label{sec:pfs}

The geometric lemma below is the key to our construction of a
sphere. For $x,y\in\R^d$, write $y\preceq x$ if for each
$i=1,2,\ldots,d$ we have $|y_i|\leq |x_i|$ and $x_i y_i\geq 0$
(equivalently, $y$ lies in the closed cuboid with opposite
corners at $0$ and $x$).  For a bond $e\in\L_d$, recall that
$\pi(e)\in\Pi_d$ is the unique plaquette that intersects it.

\begin{prop}\label{boundary}  Let $d\geq 2$.
Suppose $K\subset\Z^d$ is a finite set of sites containing $0$,
with the property that, if $x\in K$, then every $y\in\Z^d$ with
$y\preceq x$ lies in $K$. Let
\begin{equation}\label{def-s}
S:=\big\{\pi(e): e\ \text{\rm is a bond with
exactly one endvertex in }K\big\}.
\end{equation}
Then $[S]$ is a sphere with $0$ in its inside.
\end{prop}

\begin{proof}
Let $U:=\bigcup_{x\in K} Q(x)$ be the union of the unit cubes
corresponding to $K$.   Note that $[S]$ is the topological
boundary of $U$ in $\R^d$.  Let
$\Sigma:=\{z\in\R^d:\|z\|_2=1\}$ be the unit sphere; we will
give an explicit homeomorphism between $[S]$ and $\Sigma$.

We claim first that $U$ is strictly star-shaped, which is to
say: if $x\in U$ then the line segment $\{\alpha
x:\alpha\in[0,1)\}$ is a subset of the topological interior of
$U$ (i.e., of $U\setminus [S]$). To check this, suppose without
loss of generality that $x$ is in the non-negative orthant
$[0,\infty)^d$. By the given properties of $K$, the open cuboid
$$H:=\prod_{i=1}^d \big(-\tfrac12\,,\, x_i\vee \tfrac12\big)$$
is a subset of $U$ (here it is important that the origin is at
the centre of a cube, rather than on a boundary); now, $H$
clearly contains the aforementioned line segment, and the claim
is proved. In the above, $x \vee y$ denotes the maximum of $x$
and $y$.

It follows that, for any point $z\in\Sigma$, the ray $\{\alpha
z:\alpha\in[0,\infty)\}$ has exactly one point of intersection
with $[S]$.  Denote this point of intersection $f(z)$.  Clearly
$f$ is a bijection from $\Sigma$ to $[S]$; we must prove that
it is a homeomorphism.  Since $\Sigma$ and $[S]$ are compact
metric spaces, it suffices to express them as finite unions
$\Sigma=\bigcup_{j=1}^r X_j$ and $[S]=\bigcup_{j=1}^r Y_j$,
where the $X_j$ and $Y_j$ are compact, and such that $f$
restricted to $X_j$ is a homeomorphism from $X_j$ to $Y_j$, for
each $j$. This is achieved by taking $\{Y_1,\ldots ,Y_r\}$
equal to the set of plaquettes $S$.  Any plaquette in $\Pi_d$
is a subset of some $(d-1)$-dimensional affine subspace
(hyperplane) of $\R^d$ that does not pass through $0$ (here the
offset of $\tfrac12$ is again important) and it is elementary
to check that the projection through $0$ from such a subspace
to $\Sigma$ is a homeomorphism to its image in $\Sigma$.

Finally, we must check that $0$ lies in the inside of the sphere
$[S]$; this is clear because $0\in U\setminus [S]$, and any
unbounded path in $\R^d$ starting from $0$ must leave $U$ at
some point, and thus must intersect $[S]$.
\end{proof}

The next lemma is closely related to a recent result on random
surfaces in \cite{ddghs}.  Consider the bond percolation model
with parameter $p$ on $\L_d$.   By a {\dof path} we mean a
self-avoiding path comprising sites in $\Z^d$ and bonds in
$\L_d$. Recall that $\sigma(k)$ is the number of paths starting
at the origin and having $k$ edges, and
$$
\mu_d:=\lim_{k\to\infty} \sigma(k)^{1/k}
$$
is the
connective constant.  Let $0=v_0,v_1,\ldots,v_k$ be the sites
(in order) of such a path.  We call the path {\dof good} if, for
each $i$ satisfying $\|v_{i-1}\|_1<\|v_{i}\|_1$, the bond with
endpoints $v_{i-1},v_{i}$ is occupied.

\begin{lemma}\label{radius}
Let $K$ be the set of sites $u\in\Z^d$ for which there is a
good path from $0$ to $u$.  If $p<\mu_d^{-2}$ then $K$ is a.s.\
finite, and moreover,
\begin{equation}\label{bound}
\P_p(\rad K\geq r)\leq C' \alpha^r, \qquad r\geq 0,
\end{equation}
for any $\alpha\in(\mu_d p,1)$, and some $C'=C'(p,d,\alpha
)<\infty$.  If $p<(2d-1)^{-2}$ then \eqref{bound} holds with
$\alpha=p(2d-1)$ and $C'=2/[1-p(2d-1)^2]$.
\end{lemma}

\begin{proof}
Let $N(r)$ be the number of good paths that start at $0$ and
end on the $\ell^1$-sphere $\{x\in\Z^d:\|x\|_1=r\}$.  Then,
$$
\P_p(\rad K\geq r)\leq\P_p( N(r)>0)\leq
\E_p N(r).$$ For any path $\pi$ with vertices
$0=v_0,v_1,\ldots,v_k=u$ with $\|u\|_1=r$, let
$$A:=\#\{i:\|v_i\|_1>\|v_{i-1}\|_1\};
\qquad B:=\#\{i:\|v_i\|_1<\|v_{i-1}\|_1\}
$$
be respectively the
number of steps Away from, and Back towards, the origin $0$.
Note that $k=A+B$, and $\|u\|_1=A-B$.  Thus, the probability
that $\pi$ is good is $p^A$, while the number of possible paths
having given values of $A$ and $B$ is at most $\sigma(A+B)$.
Hence,
\begin{equation}\label{sum}
\E_p N(r)\leq \sum_{\substack{A,B\geq 0:\\A-B=r}} \sigma(A+B) p^A
=\sum_{B\geq 0} \sigma(2B+r)p^{B+r}.
\end{equation}

For any $\epsilon>0$, we have $\sigma(k)\leq (\mu+\epsilon)^k$
for $k$ sufficiently large, where $\mu=\mu_d$. Therefore,
\eqref{sum} is at most
$$\sum_{B\geq 0} (\mu+\epsilon)^{2B+r}p^{B+r}=
\frac{[(\mu+\epsilon)p]^r}{1-(\mu+\epsilon)^2p}$$
provided $(\mu+\epsilon)^2p<1$ and $r$ is sufficiently large;
thus we can choose $C'$ so that the required bound
\eqref{bound} holds for all $r\ge 0$.

The claimed explicit bound in the case $p<(2d-1)^{-2}$ follows
similarly from \eqref{sum} using $\sigma(k)\leq
(2d)(2d-1)^{k-1}\leq 2 (2d-1)^{k}$.
\end{proof}

\begin{proof}[Proof of Theorem \ref{sphere}]
Couple the bond and plaquette percolation models by considering
$e$ to be occupied if and only $\pi(e)$ is occupied. Let
$p<\mu_d^{-2}$ and let $K$ be the random set of sites $u$ for
which there exists a good path from $0$ to $u$. By Lemma
\ref{radius}, $K$ satisfies \eqref{bound} with the given
constants $\alpha$, $C'$.

Since a good path may always be extended by a step towards the
origin (provided the new site is not already in the path), $K$
satisfies the condition that $x\in K$ and $y\preceq x$ imply
$y\in K$. Therefore, Proposition~\ref{boundary} applies.  If $e$
is a bond with exactly one endvertex in $K$, then by the
definition of a good path, it is the end closer to $0$ that is
in $K$, and $e$ must be unoccupied. Therefore, all plaquettes in
the set $S$ in \eqref{def-s} are unoccupied. Finally, the tail
bound in \eqref{bound} implies the bound in
Theorem~\ref{sphere} because $\rad [S]\leq \rad K+d/2$.
\end{proof}

\begin{proof}[Proof of Corollary \ref{ent}]
Couple the bond and plaquette models as usual. Let
$p<\mu_3^{-2}$, and let $S$ be the set of plaquettes from
Theorem \ref{sphere}. The sphere $[S]$ intersects no occupied
bond and has $0$ in its inside, hence it has $[E]$ in its
inside.
\end{proof}

\section{Critical values}\label{sec:crit}

We consider next the critical value of $p$ for the event of
Theorem~\ref{sphere}, namely the event that there exists a
finite set $S$ of unoccupied plaquettes whose union $[S]$ is a
sphere with $0$ in its inside. In so doing, we shall make use
of the definition of a good path from Section \ref{sec:pfs}. We
shall frequently regard $\Z^d$ as a graph with bond-set $\L_d$.
A directed path of $\Z^d$ is called {\dof oriented} if every
step is in the direction of increasing coordinate-value.

Let $d\ge 2$, and (as after Theorem \ref{sphere})
declare a bond of $\L_d$ to be {\dof occupied}
with probability $p$. We write $v \too w$
(\resp, $v \too \oo$) if there exists an
oriented occupied path from $v$ to $w$ (\resp, an
infinite oriented occupied path from $v$).
Let $\pco(d)$ denote the critical probability of oriented percolation on $\Z^d$.

Let
$\tg(p)$ be the probability that there exists an infinite good path
beginning at the origin. Since $\tg$ is a non-decreasing function,
we define a critical value
$$
\pg:= \sup\{p: \tg(p)=0\}.
$$
Note that
\begin{equation}\label{h2}
\mu_d^{-2} \le \pg \le \pco(d).
\end{equation}
That $\mu_d^{-2} \le \pg$
follows by Theorem \ref{sphere}; the second inequality $\pg\le
\pco(d)$ holds since every oriented occupied path from $0$ is
necessarily good.

For $x=(x_1,x_2,\dots,x_d)\in\Z^d$, let
$$
s(x):=\sum_{i=1}^d x_i,
$$
and let
\begin{gather*}
H_n:=\{x\in \Z^d: s(x)=n\};\ H:=\{x\in\Z^d: s(x)\ge 0\};\\
\ H_+ := H\sm H_0.
\end{gather*}

A finite or infinite path $v_0,v_1,\dots$ is called {\dof\adm}
if, for each $i$ satisfying $s(v_{i-1}) < s(v_i)$, the bond
with endpoints $v_{i-1},v_i$ is occupied. If there exists an
admissible path from $x$ to $y$, we write $x \toa y$; if such a
path exists using only sites in some set $S$, we write $x\toa
y$ in $S$.

Let $e:=(1,1,\dots,1)$ and
$$
R := \sup\{n \ge 0: 0 \toa ne\}.
$$
Let $\ta(p):=\P_p(R=\oo)$, with associated critical value
$$
\pa:=\sup\{p: \ta(p)=0\}.
$$
By the definition of admissibility,
\begin{equation}\label{omh5}
\ta(p)=\P_p(\forall x\in \Z^d,\, 0\toa x),
\end{equation}
and indeed the associated events are equal.

If $x \toa y$ by an admissible path using only sites of $H_+$
except possibly for the first site $x$, we write $x\toah y$. We
write $x\toah \oo$ if $x$ is the endvertex of some infinite
admissible path, all of whose vertices except possibly $x$ lie
in $H_+$. Let $ \tah(p) =\P_p(0 \toah \oo),$ with associated
critical value
$$
\pah:= \sup\{p: \tah(p)=0\}.
$$
Also define the orthant
$$
K:= \{x\in\Z^d: x_i \ge 0 \text{ for all } i\}.
$$
Let $\ta^K(p)$ be the probability of an infinite admissible
path in $K$ starting at $0$, and let $\pa^K$ be the associated
critical value. Since $K \subseteq H_+ \cup \{0\}$, we have
$\tah\ge \ta^K$ and $\pah\le \pa^K$.

\begin{thm}\label{thm:critpt}
For $d \ge 2$ we have $\pa \le \pah=\pa^K$.
\end{thm}

Since every admissible path in the orthant $K$ is good, we have
that $\pg\le \pa^K$, and therefore $\pg\le \pah$ by Theorem
\ref{thm:critpt}. We pose two questions.
\begin{question}\label{q1} For $d\geq 3$, is it the case that
$\pa=\pah$?
\end{question}
\begin{question}\label{q2}
For $d\geq 3$, is it the case that $\pg=\pah$?
\end{question}
These matters are resolved as follows when $d=2$.

\begin{thm}\label{thm:2d}
For $d=2$ we have $\pg=\pah=\pa=1-\pco(2)$.
\end{thm}

In advance of the proofs, we present a brief discussion of Open
Question~\ref{q1} above. By Proposition \ref{lem4}(b) below, one has that
$$
\lim_{n\to\oo} \P_p(0 \toah H_n) \begin{cases} =0 &\text{ if } p<\pah,\\
>0 &\text{ if } p>\pah.
\end{cases}
$$
Now,
\begin{align*}
\P_p(0 \toa ne) &\le \sum_{x\in H_0} \P_p(x \toah ne)\\
&=\sum_{x\in H_0} \P_p(0 \toah ne-x)
= \sum_{x\in H_n} \P_p(0\toah x).
\end{align*}
If one could prove that
$$
\sum_{x\in H_n} \P_p(0\toah x) \to 0\q\text{as}\q n\to\oo,
$$
whenever $p<\pah$, it would follow by Theorem \ref{thm:critpt}
that $\pa=\pah$. This is similar to the percolation problem
solved by Aizenman--Barsky and Menshikov, \cite{AB,Men1,Men0}
(see also \cite[Chap.\ 5]{g2}). It seems possible to adapt
Menshikov's proof to prove an exponential-decay theorem for
{\em admissible} paths (as in \cite[Thm 2]{gh4}), but perhaps not for admissible
connections {\em restricted to} $H$.

The proofs of the two theorems above will make use of the
next proposition. Let $d \ge 2$ and $0\le a < b \le \oo$. Define the
cone $K_{a,b}$ to be the set of sites $x=(x_1,x_2,\dots,x_d)\in
\Z^d$ satisfying:
\begin{equation}\label{h6}
x_1\ge 0,\q\text{and } ax_1 \le x_j\le bx_1\text{ for } j=2,3,\dots,d.
\end{equation}
The subgraph of $\Z^d$ induced by $K_{a,b}$ comprises a unique infinite component, denoted
$I(K_{a,b})$, together with a finite number of finite components.
This may be seen as follows. The set of $x \in K_{a,b}$ with $x_1=k$
is the set $S_k$ of all integer-vectors
belonging to $\{k\}\times[ak,bk]^{d-1}$. Each  $S_k$ is connected, and,
for sufficiently large $k$,  $S_k$ is connected to $S_{k+1}$.

\begin{prop}\label{lem4}
Let $d \ge 2$ and $0\le a \le 1 < b \le \oo$.
\begin{letlist}
\item If $p > \pco(d)$, then
$$
\P_p(K_{a,b} \text{ \rm contains some infinite oriented occupied path}) = 1,
$$
and for all $v\in I(K_{a,b})$,
$$
\P_p(v\too \oo \text{ \rm in } K_{a,b}) >0.
$$
\item If $p > \pah$, then
$$
\P_p(K_{a,b} \text{ \rm contains some infinite admissible path}) = 1,
$$
and for all $v\in I(K_{a,b})$,
$$
\P_p(v\toa \oo \text{ \rm in } K_{a,b}) >0.
$$
\end{letlist}
\end{prop}

The remainder of this section is set out as follows. First, we
deduce Theorems \ref{thm:critpt} and \ref{thm:2d} from Proposition
\ref{lem4}. The proof of Proposition \ref{lem4} is not presented in
this paper, since it would be long and would repeat many
constructions found elsewhere. Instead, this section ends with
some comments concerning that proof.

\begin{proof}[Proof of Theorem \ref{thm:critpt}]
As noted before the statement of Theorem \ref{thm:critpt}, $K
\subseteq H_+ \cup\{0\}$, whence $\pah\le \pa^K$. Let $p>\pah$.
By Proposition \ref{lem4}(b), we have
$\ta^K(p)>0$, so that $p\ge \pa^K$. Therefore, $\pah=\pa^K$.

There is more than one way of showing $\pa\le\pah$, of which
the following is one. Let $d \ge 3$; the proof is similar when
$d=2$.  Let $p>\pah$ and fix $0<a \le 1<b<\oo$ arbitrarily. By Proposition
\ref{lem4}(b), $K_{a,b}$ contains a.s.\ some infinite
admissible path $\pi$. Any infinite path $\pi$ in $K_{a,b}$ has
the property that, for all $x\in \Z^d$, there exists $z\in \pi$
with $x \le z$ (in that $x_i\le z_i$ for every coordinate $i$).
By \eqref{omh5}, $\ta(p)>0$, so that $p\ge \pa$ and $\pah\ge
\pa$ as claimed.
\end{proof}

\begin{proof}[Proof of Theorem \ref{thm:2d}]
We shall make extensive use of two-dimensional duality.  We
call the graph $\Z^2$ the {\dof primal lattice}, and we call
the shifted graph $\Z^2+(\tfrac12,\tfrac12)$ the {\dof dual
lattice}.  Thus, the dual bonds are precisely the plaquettes of
$\Pi_2$.  Recall that a dual bond is declared occupied if and
only if the primal bond that crosses it is occupied.  For
consistency with standard terminology, we now call a dual bond
{\dof open} if and only if it is {\em unoccupied} (so a dual
bond is open with probability $q:=1-p$). We assign directions
to dual bonds as follows: a horizontal dual bond is directed
from left to right, and a vertical bond from top to bottom.

Consider the sets
\begin{align*}
D^+ &:=\{(-u,u)+(-\tfrac12,\tfrac12): u \ge 0\},\\
D^- &:=\{(u,-u)+(\tfrac12,-\tfrac12): u \ge 0\},
\end{align*}
of dual sites. The primal origin $0$ lies in some infinite admissible path in
$H^+\cup\{0\}$ if and only if no site of $D^+$ is connected by an directed
open dual path of $H$ to some site of $D^-$ (see Figure \ref{diag}). If $1-p
< \pco$ (\resp, $1-p
> \pco$) the latter occurs with strictly positive probability (\resp,
probability $0$). Therefore, $\pah = 1-\pco$.

\begin{figure}
\centering\includegraphics[width=4.3cm]{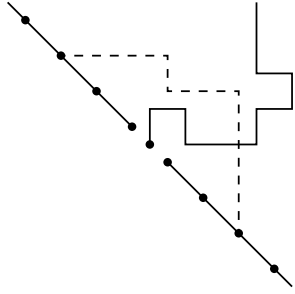}
\caption{The diagonal lines $D^+$ (top left) and $D^-$ (bottom right),
a directed dual path (dashed) in $H$ from $D^+$ to $D^-$, and a primal path
(solid) in $H^+$ from the origin (centre).}\label{diag}
\end{figure}

We show next that $\pg=\pah$. Since $\pg\le\pah$, it suffices to show that
$\pg \ge 1-\pco$. Let $p<1-\pco$, so that $1-p>\pco$. We shall prove the
required inequality $p\le \pg$. Define the set of dual sites
$$
Q:=\bigl\{(x,y)+(\tfrac12,\tfrac12):x,y\in\Z, \; x,y\geq 0\bigr\}.
$$
Let $B(k) \subseteq Q$ be given by $B(k) := ([0,k]\cap\Z)^2 +
(\tfrac12,\tfrac12)$. For $n \ge 0$, let $C_n$ be the event that there
exists a directed open dual path from $v_n:=(0,n)+(\tfrac12,\tfrac12)$
to $w_n:=(n,0)+(\tfrac12,\tfrac12)$ lying entirely within the region
$B(n) \sm B(\frac13 n)$. We claim that there exists $\b>0$ such that
\begin{equation}\label{h5}
\P_p(C_n) \ge \b, \qq n \ge 1,
\end{equation}
and the proof of this follows.

Let $V_n$ be the event that there exists a directed open dual path from $v_n$
to the line $\{(n,k)+(\frac12,\frac12): 0\le k\le n\}$ lying entirely within
the cone $\{(x,y)+(\frac12,\frac12): 0 \le n-y \le x,\ x \ge 0\}$; let $W_n$
be the event that such a path exists to $w_n$ from some site on the line
$\{(k,n)+(\frac12,\frac12): 0\le k\le n\}$, this path lying entirely within
$\{(x,y)+(\frac12,\frac12): n-x\le y < \oo,\ x \le n\}$. By Proposition \ref{lem4}(a),
$\P_p(V_n)\geq \delta$ for some $\delta=\delta(p)>0$ not depending on $n$. By reversing the directions of all
dual bonds, we see that $\P_p(W_n)=\P_p(V_n)$. On the event $V_n \cap W_n$,
there exists a directed open dual path of $B(n) \sm B(\frac13 n)$ from $v_n$
to $w_n$, and hence, by the Harris--FKG inequality,
$$
\P_p(C_n) \ge \P_p(V_n)\P_p(W_n) \geq \delta^2,
$$
as required for \eqref{h5}.

By considering corresponding events in the other three
quadrants of $\Z^2$ (with appropriately chosen
bond-orientations), we conclude that each annulus of $\R^2$
with inner (\resp, outer) $\ell^\oo$-radius $\frac13 n$
(\resp, $n+\tfrac12$) contains, with probability at least
$(1-p)^4\b^4$, a dual cycle  blocking good paths from the
origin. It follows that $p\le \pg$ as required.

Finally we show that $\pa=1-\pco$. Since $\pa\le\pah$ by
Theorem \ref{thm:critpt}, we have only to show that $\pa\ge
1-\pco$. This follows by \eqref{omh5} and the fact that, when
$q=1-p>\pco$, there exists $\P_p$-a.s.\ a doubly-infinite
directed open dual path intersecting the positive $y$-axis.
Here is a proof of the latter assertion. Let $\psi(q)$ be the
probability that there exists an infinite oriented path from
the origin in oriented percolation with density $q$. By
reversing the arrows in the fourth quadrant, the probability
that $0$ lies in a doubly-infinite directed open path of $\Z^2$
is $\psi^2$. The event
$$
J:=\{\text{there exists a doubly-infinite
open dual path}\}
$$
is a zero--one event and $\P_p(J)\ge \psi^2$, so that
$\P_p(J)=1$. Let $J^+$ (\resp, $J^-$) be the event that such a
path exists and intersects the positive (\resp, the
non-positive) $y$-axis. By reversing the directions of bonds,
we have that $\P_p(J^+) = \P_p(J^-)$. By the Harris--FKG
inequality,
$$
0=\P_p(\comp J)
= \P_p(\comp {J^+} \cap \comp{J^-})
\ge \P_p(\comp{J^+})\P_p(\comp{J^-}) = \P_p(\comp{J^+})^2,
$$
so that $\P_p(J^+)=1$.
\end{proof}

Proposition \ref{lem4} may be proved
by the dynamic-renormalization arguments developed for percolation
in \cite{bgn,gm}, for the contact model in \cite{bg},
and elaborated for directed percolation in \cite{ghiemer}. An account of
dynamic renormalization for percolation may be found in \cite{g2}.
The proof of Proposition \ref{lem4} is omitted, since it requires no novelty
beyond the above works, but extensive duplication of material therein.
The reader is directed mainly to \cite{ghiemer}, since the
present proposition involves a model in which the edge-orientations are
important. The method yields substantially more than
the statement of the proposition, but this is not developed here.

We highlight several specific aspects of the proof of Proposition \ref{lem4}, since they involve minor
variations on the method of \cite{ghiemer}.
First, since Proposition \ref{lem4} is concerned with admissible
paths in sub-cones of the orthant $K$, we require a
straightforward fact about oriented percolation on sub-cones of $\Z^2$,
namely that the associated critical probability is strictly
less than $1$. This weak statement leads via renormalization to
the stronger Proposition \ref{lem4}.
The following lemma is slightly stronger than
the minimum needed for the proof of Proposition
\ref{lem4}, and the proof is given at the end of this section.

\begin{lemma}\label{lem3}
Let $0 \le a < b \le \oo$, and let $K_{a,b}$ be the cone of $\Z^2$
containing all sites $(x,y)$ with $ax \le y \le bx$ and $x \ge 0$.
There exists $\eps=\eps_{a,b}>0$ such that: if $p>1-\eps$,
$$
\P_p(K_{a,b} \text{ \rm contains some infinite oriented occupied path})=1,
$$
and for all $v\in I(K_{a,b})$,
$$
\P_p(v\too\oo \text{ \rm in } K_{a,b}) > 0.
$$
\end{lemma}

Our second remark is concerned with  part (a) of  Proposition \ref{lem4}.
It is proved in \cite{ghiemer} that, when $p>\pco(d)$, there exists
a.s.\ an infinite oriented occupied path within some two-dimensional
slab of $\Z^d$ of the form $S \times[-A,A]^{d-2}$  for some $A<\oo$ and some sub-cone $S$
of the orthant $[0,\oo)^2$. In order to build such a path within $K_{a,b}$, one adapts
the construction of \cite{ghiemer} using the so-called `steering' arguments of
\cite{bgn,bg,g2,gm}.

Our final two remarks are concerned with part (b) of  Proposition \ref{lem4}.
For $y \in H$, write
$$
\down y = \{x\in H_0 : x\le y\},
$$
where $x\le y$ means that no coordinate of $x$ exceeds that of $y$. Let $J_l =\{x\in H_0: \|x\|_\oo \le l\}$.
The box $B_{L,K}$ of \cite[Sect.\ 4]{ghiemer} is replaced by
$$
B_{l,k} := \{y\in H: s(y) \le k,\, \down y \subseteq J_l \},
$$
with an amended version of \cite[Lemma 4.1]{ghiemer}.
Finally, the current proof uses the technique known as `sprinkling',
in a manner similar to that of the proofs
presented in  \cite{gm} and \cite[Sect.\ 7.2]{g2}.

\begin{proof}[Proof of Lemma \ref{lem3}]
Let $r=r_2/r_1$, $s=s_2/s_1$ be rationals (in their minimal representation
with positive integers $r_i$, $s_i$) satisfying $a< r< s<
b$.  Since $r<s$, we may choose a rational $\b/\a$ such that
$$
\frac{s_1}{r_1} < \frac{\b}{\a} < \frac{s_2}{r_2}.
$$
We set $R=(\b r_1,\b r_2)$, $S=(\a s_1,\a s_2)$, and note that
$\b r_1 > \a s_1$ and $\b r_2 < \a s_2$. Let $\pi$ be the
oriented subgraph of $\Z^2$ comprising the following:
\begin{letlist}
\item the horizontal oriented path from $0$ to $(\a s_1,0)$,
\item the vertical oriented path from $(\a s_1,0)$ to $S$,
\item the horizontal oriented path from $(\a s_1, \b r_2)$ to $R$.
\end{letlist}
See Figure \ref{fig2}.

\begin{figure}
\centering{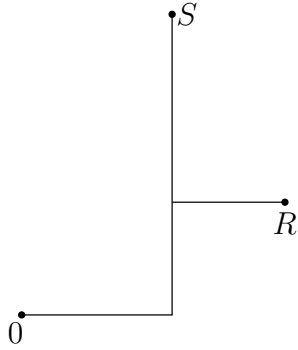}
\caption{The subgraph $\pi$ connecting $0$ to $R$ and to $S$.}\label{fig2}
\end{figure}

We may consider $I(K_{a,b})$ as the subgraph of $\Z^2$ induced
by the corresponding site-set.
Since $\pi$ is finite, we may choose $v\in
I(K_{a,b})$ such that $v + \pi$ is a subgraph of $I(K_{a,b})$, and we consider the set $v_{i,j} := v+iR+jS$, $i,j\ge
0$, of sites of $K_{a,b}$.   Let
$\a=p^N$ where $N=|\pi|$. By the definition of $K_{a,b}$,
each $v_{i,j}+\pi$ is a subgraph of $I(K_{a,b})$, and we declare $v_{i,j}$ {\dof black} if every bond in
$v_{i,j}+\pi$  is open. Note by the construction of the set $\pi$ that the
states of different sites $v_{i,j}$ are independent.

If $1-\a < (1-\pco)^2$, the set of black vertices dominates
(stochastically) the set of sites $w$ of a supercritical
oriented percolation model with the property that both bonds
directed away from  $w$ are open. The claims of the lemma
follow by standard properties of oriented percolation.
\end{proof}

\section*{Acknowledgements}
GRG acknowledges support from Microsoft Research during his
stay as a Visiting Researcher in the Theory Group in Redmond.
This work was completed during his attendance at a programme at
the Isaac Newton Institute, Cambridge.

\bibliographystyle{amsplain}
\bibliography{sphere}

\end{document}